\documentclass[11 pt]{amsart}
\usepackage{graphicx}
\usepackage[hidelinks]{hyperref}
\usepackage{color}
\usepackage{amssymb,amsmath,amsfonts,latexsym}
\usepackage{epstopdf}
\usepackage{mathrsfs}
\usepackage{tikz}
\usetikzlibrary{arrows,chains,matrix,positioning,scopes}

\usepackage{upquote}
\usepackage{nicefrac}
\usepackage{amsthm}
\usepackage{float,lscape}
\usepackage{hyperref}
\hypersetup{
     colorlinks   = true,
     citecolor    = blue
}

\usepackage{tikz}
\usepackage{tikz,fullpage}
\usetikzlibrary{arrows,chains,matrix,positioning,scopes}
\makeatletter
\tikzset{join/.code=\tikzset{after node path={%
\ifx\tikzchainprevious\pgfutil@empty\else(\tikzchainprevious)%
edge[every join]#1(\tikzchaincurrent)\fi}}}
\makeatother
\tikzset{>=stealth',every on chain/.append style={join},
         every join/.style={->}}
\tikzstyle{labeled}=[execute at begin node=$\scriptstyle,
   execute at end node=$]
\usepackage{tikz,fullpage}
\usetikzlibrary{patterns}
\numberwithin{equation}{section}
\usepackage[margin=1.27in]{geometry}
\usetikzlibrary{arrows,%
                petri,%
                topaths}%
\usepackage{tkz-berge}
\usepackage{tikz-cd}

\usetikzlibrary{patterns}
\usepackage{enumerate}
\usepackage{ulem}
\numberwithin{equation}{section}
\usepackage[margin=1.27in]{geometry}
\usetikzlibrary{arrows,%
                petri,%
                topaths}%
\usepackage{tkz-berge}
\usepackage[position=top]{subfig}
\usetikzlibrary{%
  matrix,%
  calc,%
  arrows%
}
\newtheorem{thm}{THEOREM}[section]

\newtheorem{cor}[thm]{COROLLARY}

\newtheorem{lemma}[thm]{LEMMA}

\begin{document}

\title{Hochschild and Cyclic homology of the quantum Kummer Spaces}
\author{Safdar Quddus}

\date{\today}
 
\let\thefootnote\relax

\begin{abstract}
We study the quotient space obtained by the flip action on the quantum n-tori. The Hochschild, cyclic and periodic cyclic homology are calculated.
\end{abstract}
\maketitle
\section{Introduction}

Spanier \cite{S} studied the Kummer (non-smooth)manifolds obtained by the action of $\mathbb Z_2$ on the $2n$ dimensional torus. He concluded that the space is homeomorphic to $\mathbb R \mathbb P^{2n-1}$. It has $2^{2n}$ double points and is simply connected with vanishing odd homology. Alternatively, a link of fixed-points after the quotient is homeomorphic to $\mathbb R \mathbb P^{2n-1}$. The non-commutative geometry currently does not have a well defined notion of ``non-commutative knots/links" but we shall see that homologically the $\mathbb Z_2$ quotient of the $n$-quantum torus $\mathscr A_\Theta$ is similar to the Kummer manifold/variety. The dimension of cyclic homology is the same as the Betti numbers for the classical Kummer manifolds.\\

The Hochschild homology for these orbifolds for the case $n=2$ was done in \cite{O} \cite{B} and \cite{Q}. We here are inspired by the proof of \cite{Q} and extending the methodology there into higher dimensions. It maybe noted that the periodic cyclic homology of $\mathcal A_\Theta \rtimes \mathbb Z_2$, the noncommutative smooth torus $\mathbb Z_2$ computed in a recent work \cite{CTY} matches in dimension to what we have calculated in this article for the quantum/algebraic noncommutative torus with $\mathbb Z_2$ action. It is also expected that for ``sufficiently" good $\Theta$, the $\mathbb Z_2$ quotient of the non-commutative smooth $n$-torus $\mathcal A_\Theta$ shares similar Hochschild homological property but even for $n=2$, such a computation was tricky\cite{C}. What is known rather is that the odd periodic homology vanishes which does hint the vanishing of odd Hochschild homology. Other than having a striking similarity with the smooth quotients, the Hochschild homology of the quantum tori themselves have been studied \cite{W1} and \cite{T}. Readers can refer to \cite[Page 353]{BRT} for the comparison table therein for various homological and algebraic properties between the smooth non-commutative 2-torus and the quantum/algebraic non-commutative 2-torus.\\

\section{Statement}
\begin{thm}
Let $\Theta$ be a skew symmetric $n \times n$ matrix such that its entries are unimodular but none are roots of unity.\\
a) $H_0(\mathscr A_\Theta \rtimes \mathbb Z_2) \cong \mathbb C^{2^{n} +1} \text{and}$\\
b) $ H_\bullet(\mathscr A_\Theta \rtimes \mathbb Z_2) \cong \mathbb C^{\binom{n}{\bullet}} \text{ for } \bullet = 2k \text{ for some } k>0 \text{ and } 0 \text{ otherwise }.$
\end{thm}

\begin{cor}
 $dim_{\mathbb C}(HC_{\bullet}(\mathscr A_\Theta \rtimes \mathbb Z_2)) = \displaystyle\sum_{2k \leq \bullet} \binom{n}{2k} + {2^{n}}$ for $\bullet$ even, $0$ otherwise.

\end{cor} 

\begin{cor}
The periodic homology are as follows:\\
a) $HP_{even}(\mathscr A_\Theta \rtimes \mathbb Z_2) \cong \mathbb C^{3 \cdot 2^{n-1}} \text{and}$\\
b) $ HP_{odd}(\mathscr A_\Theta \rtimes \mathbb Z_2) = 0.$
\end{cor}

\section{Strategy of the Proof}

We shall study the Hochschild homology using the paracyclic spectral decomposition of the homology of the crossed product algebra. \cite{GJ}
$$ H_\bullet(\mathscr{A}_\Theta \rtimes \mathbb Z_2) \cong H_\bullet(\mathscr{A}_\Theta)^{\mathbb Z_2} \oplus H_\bullet({}_{-}\mathscr{A}_\Theta)^{\mathbb Z_2},$$
where ${}_{-}\mathscr{A}_\Theta$ is the algebra $\mathscr{A}_\Theta$ with (left) $\mathbb Z_2$-twisted $\mathscr{A}_\Theta^e$ module structure. 

Hence our proof will investigate each of the summands in the above decomposition by firstly understanding the associated Hochschild homology and then locating the $\mathbb Z_2$ invariant cycles. We shall use Nest's resolution(which is similar to Wambst's resolution for quantum symmetric algebras\cite{W2}) and also Connes' resolution as and when we find suitable.
\section{Nest's Koszul Resolution Revisited}

Nest \cite{N} introduced a Koszul resolution for higher dimensional non-commutative tori. We briefly describe his resolution below.\\

The algebra $\mathscr A_\Theta$ for a skew symmetric complex matrix $\Theta$ is generated by unitaries $\nu_1, \nu_2, \dots , \nu_n$, satisfying the commutation relations
$$ \nu_i \nu_j = \lambda_{ij} \nu_j \nu_i, \text{ for } 1 \leq i,j \leq n $$
such that $|\lambda_{ij}|=1$.\\
The enveloping algebra $\mathscr A_\Theta^e$ is the algebraic tensor of $\mathscr A_\Theta$ and its opposite algebra.
$$\mathscr A_\Theta^e = \mathscr A_\Theta \otimes \mathscr A_\Theta^{op}.$$
An element of $\mathscr A_\Theta^e$ is denoted by $ a \otimes b^o $ for $a \in \mathscr A_\Theta$ and $b^o \in \mathscr A_\Theta^{op}$.  
We set $V = \mathbb C^n$ with orthonormal basis $e_1, e_2, \dots, e_n$. We have the standard bar resolution of $\mathscr A_\Theta$ is given as below:

$$ \cdots \xrightarrow{b} \Lambda_s(\mathscr A_\Theta) \xrightarrow{b} \Lambda_{s-1}(\mathscr A_\Theta) \xrightarrow{b} \cdots \xrightarrow{b} \mathscr A_\Theta^e \xrightarrow{\epsilon} \mathscr A_\Theta,$$
where
$ \Lambda_s(\mathscr A_\Theta) = \mathscr A_\Theta^e \otimes \mathscr A_\Theta^{\otimes s}$, $\epsilon$ is the augmentation map and $b : \Lambda_n(\mathscr A_\Theta) \rightarrow \Lambda_{n-1}(\mathscr A_\Theta)$;
$$a_0 \otimes a_1 \otimes \cdots \otimes a_n \mapsto \displaystyle\sum_{i=0}^{n-1}(-1)^i a_0 \otimes \cdots a_i a_{i+1} \cdots a_n + (-1)^n a_n a_0 \otimes a_1 \otimes \cdots a_{n-1}.$$
We set $E_s = \mathscr A_\Theta^e \otimes \Lambda^s V$ and consider the following maps:

$$h_s : E_s \rightarrow \Lambda_s(\mathscr A_\Theta);$$
$$ 1 \otimes e_{i_1} \wedge e_{i_2} \wedge \cdots \wedge e_{i_n} \mapsto \displaystyle\sum_{ \sigma \in S_s} sgn(\sigma) (\nu_{\sigma(i_1)}\nu_{\sigma(i_2)} \cdots \nu_{\sigma(i_s)})^{-1} \otimes \nu_{\sigma(i_1)} \otimes \nu_{\sigma(i_2)} \otimes \cdots \otimes \nu_{\sigma(i_s)}.$$

$$\alpha_s : E_s \rightarrow E_{s-1};$$
$$ 1 \otimes e_{i_1} \wedge e_{i_2} \wedge \cdots \wedge e_{i_n} \mapsto \displaystyle\sum_{k=1}^{s} (-1)^k (1 - \nu_{i_k}^{-1} \otimes \nu_{i_k}^{o}) \otimes e_{i_1} \wedge e_{i_2} \wedge \cdots \wedge \hat{e_{i_k}} \wedge \cdots \wedge e_{i_n} .$$

$$ k_s : \Lambda_s(\mathscr A_\Theta) \rightarrow E_s;$$
$$ k((\nu_{\pi_1}\nu_{\pi_2} \cdots \nu_{\pi_s})^{-1} \otimes \nu_{\pi_1} \otimes \nu_{\pi_2} \otimes \cdots \otimes \nu_{\pi_s}) = $$$$ \displaystyle\sum_{i_1 > i_2> \dots i_s} \rho_{i_1}((\nu_{\pi_1})^{-1} \otimes (\nu_{\pi_1}))  \wedge \rho_{i_2}((\nu_{\pi_2})^{-1} \otimes (\nu_{\pi_2})) \wedge \dots \wedge \rho_{i_s}((\nu_{\pi_s})^{-1} \otimes (\nu_{\pi_s})).  $$

Where for $E_s \cong \mathscr A_\Theta \otimes \Lambda^s V \otimes \mathscr A_\Theta$ has a graded product structure\cite{N}, for $\pi=(\pi_1, \pi_2, \dots \pi_s) \in \mathbb Z^s$, $\nu^{\pi} := \nu_1^{\pi_1} \nu_2^{\pi_2} \dots \nu_s^{\pi_s}$ and $\rho_i : \Lambda_1(\mathscr A_\Theta) \rightarrow E_1$ as defined as below

Note: The formula for $\rho_i((\nu^{\pi_1})^{-1} \otimes \nu^{\pi_1})$ in \cite{N}[Page 1050] has a misprint and the correct formula, which we shall use in our study is as follows:
$$\rho_i((\nu^{\pi})^{-1} \otimes \nu^{\pi})= (\nu^\pi|_{>i})^{-1} (\displaystyle \sideset{}{'} \sum\limits_{s=0}^{\pi_i -1} \nu_i^{-k} \otimes e_i \otimes \nu_i^k )(\nu^\pi|_{>i}).$$
where 

$\displaystyle \sideset{}{'} \sum\limits_{i=0}^{n}= \begin{cases}
\displaystyle \sum\limits_{i=0}^n & \text{ for } n \geq 0,\\
0  & \text{ for } n = -1, \\
-\displaystyle \sum\limits_{n+1}^{-1} & \text{ for } n < -1. \end{cases}$\\
and $\nu^\pi|_{> {p-1}} := \nu_{p}^{\pi_p} \dots \nu_{n}^{\pi_n}$.

Though Nest gave this resolution for smooth non-commutative $n$-tori $\mathcal A_\Theta$, but it is also a resolution of the quantum tori $\mathscr A_\Theta$, the proof is easy and similar to the proof that Connes' resolution is a resolution for quantum 2-torus.\cite{Q} 
\section{Invariant cycles, $H_\bullet(\mathscr A_\Theta)^{\mathbb Z_2}$} \par

Using the Nest resolution we can easily compute $H_0(\mathscr A_\Theta)$ explicitly, they zeroth cocycles are the $\mathscr{A}_\Theta / im(\alpha_1)$, where
$$ \alpha_1(a^i \otimes 1 \otimes e_i) = a^i\otimes (1- \nu_i^{-1} \otimes {\nu_i}^{o})=a^i - \nu_i^{-1} a^i \nu_i.$$
Clearly, the zeroth Hochschild homology $H_0(\mathscr{A}_\Theta)$ is generated by the equivalence class of elements supported at $a_{\bar{0}}$. These scalars are invariant under $\nu_i \mapsto \nu_i^{-1}$, hence $H_0(\mathscr A_\Theta)^{\mathbb Z_2} = \mathbb C$.

To compute $H_\bullet(\mathscr{A}_\Theta)^{ \mathbb Z_2}$ for $\bullet > 1$ we shall firstly observe $H_\bullet(\mathscr{A}_\Theta)$ using the Nest's Koszul resolution and then locate the invariant cycles. Wambst \cite{W1} computed these $k$-Hochschild cycles of the quantum tori $\mathscr A_\Theta$, they were generated by elements $\{ (x^\pi)^{-1} \otimes x^\pi\}_{\pi \in  \{0,1\}^{n}}$ with $|\pi| = k$. But in this article, we shall restrict ourselves with the notation of Nest \cite{N}.\\

Let us consider $E_s$ in the Nests' Koszul resolution, using the following map it is straightforward to see that $H_s(\mathscr A_\Theta)$ is generated by elements $a_{\bar{0}} \otimes e_{i_1} \wedge \cdots \wedge e_{i_s}$ where $\{i_1, \dots , i_s\} \subset \{1,2, \dots, n\}$. We want to locate the $\mathbb Z_2$ Hochschild $k$-cycle using the Koszul resolution of $\mathscr A_\Theta$.
$$ (1 \otimes d_s) : \mathscr A_\Theta \otimes E_s \rightarrow \mathscr A_\Theta \otimes E_{s-1}.$$

Let $ a_{\bar{0}} \otimes 1 \otimes e_{i_1}\wedge \cdots \wedge e_{i_s} \in ker(1\otimes d_s)$, to check if it is invariant under $\mathbb Z_2$ action we push the cycle into the bar complex using the map $h_s$.

\begin{center}
\begin{tikzcd}
   \cdots \arrow{r} &  E_2\arrow{r}{\alpha_2}\arrow[harpoon]{d}{k_2} & {E}_1\arrow{r}{\alpha_1}\arrow[harpoon]{d}{k_1}  & E_0\arrow{r}{\alpha_0}\arrow[harpoon]{d}{k_0 = id} & \mathscr A_\Theta\arrow{r}\arrow{d}{\cong} & 0 \\
    \cdots \arrow{r} & \mathscr A_\Theta^{\otimes 4}\arrow{r}{b'}\arrow[harpoon]{u}{h_2}\arrow{r} & \mathscr A_\Theta^{\otimes 3}\arrow[harpoon]{u}{h_1}\arrow{r}{b'} & \mathscr A_\Theta^{\otimes 2}\arrow{r}{b'}\arrow[harpoon]{u}{h_0=id}\arrow{r} & \mathscr A_\Theta\arrow{r} & 0
\end{tikzcd}
\end{center}

$$(1 \otimes h_s)(a_{\bar{0}}\otimes 1 \otimes e_{i_1} \wedge \cdots \wedge e_{i_s}) = a_{\bar{0}} \otimes \displaystyle \sum\limits_{\sigma  \in S_s} sgn(\sigma) (\nu_{\sigma(i_1)} \dots \nu_{\sigma(i_s)})^{-1} \otimes \nu_{\sigma(i_1)} \otimes \cdots \otimes \nu_{\sigma(i_s)}.$$ 
Now,
$$ (1 \otimes k_s)(a_{\bar{0}} \otimes {\displaystyle \sum\limits_{\sigma  \in S_s} sgn(\sigma) (\nu_{\sigma(i_1)}^{-1} \dots \nu_{\sigma(i_s)}^{-1})^{-1} \otimes \nu_{\sigma(i_1)}^{-1} \otimes \cdots \otimes \nu_{\sigma(i_s)}^{-1}})$$ $$= sgn(\psi) a_{\bar{0}} \otimes  k_s((\nu_{\psi(i_1)}^{-1} \dots \nu_{\psi(i_s)}^{-1})^{-1} \otimes \nu_{\psi(i_1)}^{-1} \otimes \cdots \otimes \nu_{\psi(i_s)}^{-1})$$

Where $\psi \in S_k$ is the permutation such that $\psi(i_1) > \psi(i_1) > \cdots \psi(i_s)$. We have used the fact that $\rho_i((\nu^\pi)^{-1} \otimes \nu^\pi) = 0$ if $(\pi)_i =0$.

$$\rho_{\psi_j}(\nu_{\psi(i_j)} \otimes \nu_{\psi(i_j)}^{-1}) = -(\nu_{\psi_j} \otimes e_{\psi_j} \otimes \nu_{\psi_j}^{-1}).$$
Therefore,
 $$ (1 \otimes k_s)(a_{\bar{0}} \otimes {\displaystyle \sum\limits_{\sigma  \in S_s} sgn(\sigma) (\nu_{\sigma(i_1)}^{-1} \dots \nu_{\sigma(i_s)}^{-1})^{-1} \otimes \nu_{\sigma(i_1)}^{-1} \otimes \cdots \otimes \nu_{\sigma(i_s)}^{-1}})$$ $$= (-1)^{s} sgn(\psi) sgn(\psi)^{-1} a_{\bar{0}} \otimes 1 \otimes e_{i_1} \wedge \cdots \wedge e_{i_s}.$$

Hence for $s$ even all the $s$-cycles are $\mathbb Z_2$ invariant and for $s$ odd none are. This was exactly the case in \cite[Page 329, 331]{Q}, for the 1-cycles and the 2-cycle of the quantum 2-torus with $SL_2(\mathbb Z)$ action. We have the following:
\begin{lemma}
$H_\bullet(\mathscr A_\Theta)^{\mathbb Z_2} = \mathbb C^{\binom{n}{\bullet}}$ if $\bullet =2k$, 0 otherwise. 
\end{lemma}

\section{Twisted invariant cycles, $H_\bullet({}_{-}\mathscr A_\Theta)^{\mathbb Z_2}$}
To calculate $H_\bullet({}_{-}\mathscr A_\Theta)$ we need to consider the $\mathbb Z_2$ twisted Koszul chain complex. For $s=0$ we can explicitly see that the $\mathbb Z_2$ twisted zeroth cycle.
$$H_0({}_{-}\mathscr A_\Theta) = {}_{-}\mathscr A_\Theta / {}_{-}\alpha_1.$$ where
$$ {}_{-}\alpha_1(a^i \otimes 1 \otimes e_i) = a^i\otimes(1- \nu_i^{-1} \otimes {\nu_i}^{o})=a^i - \nu_i a^i \nu_i.$$
Hence $H_0({}_{-}\mathscr A_\Theta)$ is generated by the equivalence class of elements of the form $(a_\beta)_{\beta \in \{0,1\}^{n}}$. Therefore $H_0({}_{-}\mathscr A_\Theta) = \mathbb C^{2^n}$. Under the action of $\mathbb Z_2$, an element $(a_\beta)$ is mapped to homologous element $(a_{-\beta})$ hence $$H_0({}_{-}\mathscr A_\Theta)^{\mathbb Z_2} = \mathbb C^{n}.$$
Observe that for $a \in H_{n}({}_{-}\mathscr A_\Theta) = ker (1 \otimes d_{n})$, $ a = \nu_i a \nu_i$ for all $i$. Hence $H_{n}({}_{-}\mathscr A_\Theta) = 0$.\\
We shall proceed by induction, we shall induct on the dimension of the torus.\\

We state that for the $n$- dimensional quantum torus($n >1$), $H_\bullet({}{-}\mathscr A_\Theta)=0$ for all $0 < \bullet <n$. As we noted earlier, the above statement holds for the case $n=2$. Let us assume that for all torus of dimensions less than $n$ it holds. We shall prove that $H_\bullet({}{-}\mathscr A_\Theta)=0$ for all $0 < \bullet <n$. The proof is now divided into two cases: 
\subsection{Case I: $\bullet > 1$}
\begin{lemma}
In this case we shall prove that if $H_\bullet({}_{-}\mathscr A_\Theta)=0$ for $n-1>\bullet >0$ then $H_\bullet({}_{-}\mathscr A_\Theta^n) = 0$ for all $n>\bullet>1$. \\
\end{lemma}
\begin{proof}

We notice that $\binom{n}{s} = \binom{n-1}{s} +\binom{n-1}{s-1}$. Hence we have the following identification.\\
$$ (\pi_1, \pi_2) : {}_{-}\mathscr A_\Theta \otimes E^n_s = {}_{-}\mathscr A_\Theta \otimes  E^{n-1}_s \oplus {}_{-}\mathscr A_\Theta \otimes  E^{n-1}_{s-1} \wedge e_n$$
Therefore to show that $\mathscr A_\Theta \otimes E^n_\bullet$ is acyclic at $s$ it is enough to show that $${}_{-}\mathscr A_\Theta \otimes E^{n}_{s+1} \bigoplus {}_{-}\mathscr A_\Theta \otimes E^{n}_{s} \wedge e_n \xrightarrow{1 \otimes {}_{-}\alpha_{s+1}} {}_{-}\mathscr A_\Theta \otimes  E^{n-1}_s \bigoplus {}_{-}\mathscr A_\Theta \otimes  E^{n-1}_{s-1} \wedge e_n \xrightarrow{1 \otimes {}_{-}\alpha_{s}} {}_{-}\mathscr A_\Theta \otimes  E^{n-1}_{s-1} \bigoplus {}_{-}\mathscr A_\Theta \otimes  E^{n-1}_{s-2} \wedge e_n $$ is middle exact.\\
We notice that the map $1 \otimes {}_{-}\alpha^{n}_{s}$ does mix the direct summands. Explicitly, $$(1 \otimes {}_{-}\alpha_{s})({}_{-}\mathscr A_\Theta \otimes  E^{n-1}_s) \subset {}_{-}\mathscr A_\Theta \otimes  E^{n-1}_{s-1};$$ $$(1 \otimes {}_{-}\alpha_{s})({}_{-}\mathscr A_\Theta^{n} \otimes  E^{n-1}_{s-1}\wedge e_n) \subset {}_{-}\mathscr A_\Theta^{n} \otimes  E^{n-1}_{s-1} \bigoplus {}_{-}\mathscr A_\Theta \otimes  E^{n-1}_{s-2} \wedge e_n.$$\\
Hence for $\gamma \in ker(1\otimes  {}_{-}\alpha_s^{n})$, $\pi_2(\gamma) \in ker (1 \otimes {}_{-}\alpha^{n-1}_{s-1} ) \subset {}_{-}\mathscr A_\Theta \otimes  E^{n-1}_{s-2} \wedge e_n$. But, since by hypothesis $H_{s-1}({}_{-}\mathscr A_\Theta )=0$, there exists $\mu_2 \in {}_{-}\mathscr A_\Theta \otimes  E^{n-1}_{s-1}$ such that $$(1 \otimes {}_{-}\alpha^{n-1}_{s})(\mu_2) = \pi_2(\gamma).$$
Therefore, $\pi_2(1 \otimes {}_{-}\alpha^n_{s+1})(\mu_2 \wedge e_n)= \pi_2(\gamma)$.

 
 We are now left to prove that there exists a $\mu_1 \in {}_{-}\mathscr A_\Theta \otimes  E^{n-1}_{s}$, such that $\pi_1(1 \otimes {}_{-}\alpha^n_{s+1})(\mu_1)= \pi_2(\gamma)$. A kernel relation over ${}_{-}\mathscr A_\Theta \otimes E^{n-1}_{s-1}$ having the indices $e_{r_1} \wedge e_{r_2} \wedge \cdots \wedge e_{r_{s-1}}$ such that $r_i \neq n$ for any $i$ looks like
 
 $$\psi^{p_1} V_{p_1} + \psi^{p_2} V_{p_2} +\cdot + \psi^{p_{n-s}} V_{p_{n-s}}+  \psi^{p_n} V_{e_n}=0.$$
 
  Where $V_{e_i} := (1- \nu_i \otimes \nu_i^{-1})$, $p_i \neq e_j$ for any $i, j$ and $\psi^k \in {}_{-}\mathscr A_\Theta$. It can be observed that $\psi^{p_n} = - c^{p_1} V_{p_1} -c^{p_2} V_{p_2} - \cdots  - c^{p_{s-1}} V_{p_{n-s}}$, where $c^{p_k}$ are the coefficients of $e_{r_1} \wedge e_{r_2} \wedge \cdots \wedge e_{r_{s-1}} \wedge e_{p_k} \wedge e_n \in E^{n}_{s} \wedge e_n$. Since for $a \in {}_{-}\mathscr A_\Theta$, $ a V_ i V_n = a V_n V_i$ for all $i$, the above kernel relation is hence reduced to one of the following form
  
  $$ \psi^{' p_1} V_{p_1} + \psi^{' p_2} V_{p_2} +\cdot + \psi^{' p_{n-s}} V_{p_{n-s}}=0.$$
  Which has a solution by induction hypothesis, $H_{s}({}_{-}\mathscr A_\Theta)=0$ for the $n-1$ dimensional torus. Hence we are done. 
 
 \end{proof}
\subsection{Case II: $\bullet =1$}

We prove for this case by induction over the dimension of torus and using the techniques of  \cite{Q}. The $ker(1 \otimes d_{s}) \subset {}_{-}\mathscr A_\Theta \otimes E_s$ are represented as diagrams in the $\binom{n}{s}$-dimensional lattice space. For $\gamma \in  ker(1 \otimes d_{s})$, consider a its diagram, $Diag(\gamma)$. Without loss of generality we may assume that $Diag(\gamma)$ is a connected sub-lattice of $\mathbb Z ^{\binom{n}{s}}$ assembled  by $\binom{n}{s}$ dimensional polytopes. The case $2s=n=2$ corresponds to the quantum 2-torus. Here $s=1$, hence consider an arbitrary 1-cocycle $\gamma$ and its $Diag(\gamma) \subset \mathbb Z^n$. We shall prove that $\gamma$ is homologous to 0 in a similar way as we did in \cite{Q}. We firstly change the basis of the Koszul resolution, while the basis of Connes' Koszul resolution is $1\otimes u_1$ and $1\otimes u_2$ for the 2 dimensional case and $(1 \otimes u_{i_1}\otimes u_{i_2} \otimes \ldots  \otimes u_{i_s})$ in general, the basis for the Nest's Koszul resolution is anti symmetrised
$$ 
\left( 
(u_{i_1}u_{i_2}\ldots  u_{i_s})^{-1}\otimes u_{i_1}\otimes u_{i_2}\otimes \ldots  u_{i_s}
\right)
.$$
While Nest's resolution is computationally convenient, Connes' basis is more convenient for diagrammatic approach \cite{Q}. In this subsection, we shall consider the Generalized Connes' Koszul resolution. An element of ${}_{-}\mathscr A_\Theta$ is finitely supported in $\mathbb Z^n$, hence there exists $l >0$ such that $Diag(\gamma) \subset B_l(\bar{0})$. The hyperplanes $ x_1 = l$ and $x_1 = -l$ contain $Diag(\gamma)$ between them.\\

 A connected component of $Diag(\gamma)$ is an assembly of $n$-hypercubes with no edges, i.e.  diagram of the following form does not exist.
 \begin{center}
\begin{tikzpicture}
\draw (-0.2,0) -- (1.2,0) node(xline)[right] {};
\fill (canvas cs:x=1.2cm,y=0cm) circle (3pt);

\end{tikzpicture}
\end{center} 

A typical kernel diagram in $\mathbb Z^3$ looks like as below. It is cubes connected by the kernel relation, the bullet dots represents a non-zero element of $\mathscr A_\Theta^{\oplus n}$.
\begin{tikzpicture}

\def \dx{1};
\def \dy{2};
\def \dz{2};
\def \nbx{3};
\def \nby{3};
\def \nbz{3};

\foreach \x in {1,...,\nbx} {
    \foreach \y in {1,...,\nby} {
        \foreach \z in {1,...,\nbz} {
            \node at (\x*\dx,\y*\dy,\z*\dz) [circle, fill=black] {};
        }
    }
}

\foreach \x in {1,...,\nbx} {
    \foreach \z in {1,...,\nbz}{
        \draw (\x*\dx,\dy,\z*\dz) -- ( \x*\dx,\nby*\dy,\z*\dz);
    }
}

\foreach \y in {1,...,\nbx} {
    \foreach \z in {1,...,\nbz}{
        \draw (\dx,\y*\dy,\z*\dz) -- ( \nbx*\dx,\y*\dy,\z*\dz);
    }
}

\foreach \x in {1,...,\nbx} {
    \foreach \y in {1,...,\nbz}{
        \draw (\x*\dx,\y*\dy,\dz) -- ( \x*\dx,\y*\dy,\nbz*\dz);
    }
}

\end{tikzpicture}

\begin{lemma}
$H_1({}_{-}\mathscr A_\Theta) = 0$.
\end{lemma}
\begin{proof}
We prove by induction, let us assume that for all tori of dimension less than $n$ the first Hochschild homology vanishes. Consider $Diag(\gamma) \cap \{x_1= l\}$, we choose the $a_{l_i} \otimes e_1 \wedge e_i$ such that the projection of $(1 \otimes \alpha_2^n)(a_{l_i} \otimes e_1 \wedge e_i)$ on the hyperplane $\{x_1 = l\}$ kills $Diag(\gamma) \cap \{x= l\}$. This can be done by ordering the non-zero lattice points in the $i$th dimension and using $e_1 \wedge e_i$ with appropriate coefficients to kill them. Thus, what remains is a cycle which is not supported on $\{x_1=l\}$ and is homologous to $\gamma$.  Repeating this for each hyper plane $\{x_i = d\}, d<l$ we end up with a diagram which represents a cycle $\gamma_{d_0}$ homologous to $\gamma$ and lies entirely in $\{x_1 =d_0\}$ for some $d_0 \geq -l$. But we observe that $\gamma_{d_0}$ is also a 1-cycle for $n-1$ torus and hence homologous to 0 by induction hypothesis.  

\end{proof}

\begin{proof}[Proof of Theorem 1.1]
It follows from Lemma 4.1, Lemma 5.1 and Lemma 5.2.

\end{proof}

\section{Cyclic Homology}

Connes introduced an $S,B,I$ long exact sequence relating the Hochschild and cyclic homology of an algebra $A$,
\begin{center}
$... \xrightarrow{B} HH_n(A) \xrightarrow{I} HC_n(A) \xrightarrow{S} HC_{n-2}(A) \xrightarrow{B} HH_{n-1}(A) \xrightarrow{I}...$.\\
\end{center}
The cyclic homology 
\begin{proof}[Proof of Corollary 1.2 and 1.3]

The $\mathbb Z_2$ action on $\mathscr A_\Theta$ commutes with the map ${}_{-1}\alpha$ , we obtain the following exact sequence
\begin{center}
$... \xrightarrow{B} (HH_n({}_{-}\mathscr A_\Theta))^{\mathbb Z_2} \xrightarrow{I} (HC_n({}_{-}\mathscr A_\Theta))^{\mathbb Z_2} \xrightarrow{S} (HC_{n-2}({}_{-}\mathscr A_\Theta))^{\mathbb Z_2} \xrightarrow{B} (HH_{n-1}({}_{-}\mathscr A_\Theta))^{\mathbb Z_2} \xrightarrow{I}...$.\\
\end{center}
Using the above long exact sequence we deduce the cyclic homology of $\mathscr A_\Theta \rtimes \mathbb Z_2$[Corollary 1.2]. The $\mathbb Z_2$ invariant periodic cyclic homology for ${}_{-}\mathscr A_\Theta$ is $HP_{even}({}_{-}\mathscr A_\Theta)^{\mathbb Z_2} = \mathbb C^{2^n}$ and $HP_{odd}({}_{-}\mathscr A_\Theta)^{\mathbb Z_2} = 0$. Similarly for the untwisted case $HP_{even}(\mathscr A_\Theta)^{\mathbb Z_2}$ has dimension ${\displaystyle\sum_{ \bullet = 2k} \binom{n}{\bullet}}$, hence;

$$HP_{even}(\mathscr A_\Theta)^{\mathbb Z_2} =  \mathbb C^{2^{n-1}} \text { and } HP_{odd}(\mathscr A_\Theta)^{\mathbb Z_2} = 0.$$\\
By using the paracyclic spectral decomposition we have:
$$HP_{even}(\mathscr A_\Theta \rtimes \mathbb Z_2 ) \cong HP_{even}(\mathscr A_\Theta)^{\mathbb Z_2} \oplus HP_{even}({}_{-}\mathscr A_\Theta)^{\mathbb Z_2} = \mathbb C^{2^n} \oplus \mathbb C^{2^{n-1}} = \mathbb C^{3 \cdot 2^{n-1}}$$
and $$HP_{odd}(\mathscr A_\Theta \rtimes \mathbb Z_2 ) = 0.$$
\end{proof}

\vspace{2mm}

{\small \noindent{Safdar Quddus},\\
 Department of Mathematics, India Institute of Science, Bengaluru\\
Email: safdarquddus@iisc.ac.in .

\end{document}